\documentclass[11pt,a4paper,twoside]{article}
\usepackage{amsmath,amsthm,amssymb,amsfonts,amsbsy}

\usepackage{color, colortbl} 
\usepackage{float} 
\usepackage[margin=2.5cm]{geometry} 
\usepackage{cite}
\usepackage{authblk} 

\theoremstyle{definition}
\newtheorem{definition}{Definition}[section]

\theoremstyle{plain}
\newtheorem{thm}[subsection]{Theorem}
\newtheorem{lem}[subsection]{Lemma}
\newtheorem{prop}[subsection]{Proposition}

\newcommand{\beq}{\begin{eqnarray}}
\newcommand{\eeq}{\end{eqnarray}}

\newcommand{\beqs}{\begin{eqnarray*}}
\newcommand{\eeqs}{\end{eqnarray*}}

\newcommand{\al}{\alpha}

\allowdisplaybreaks

\usepackage{mathtools}

\title{ \bf A note on polyomino chains with extremum general sum-connectivity index}

\author{
   \large \bf Akbar Ali\footnote{Corresponding author.} , Tahir Idrees
}

\affil{ \normalsize
    { Knowledge Unit of Science,\\  University of Management \& Technology, Sialkot-Pakistan}
    \\E-mail: {\tt akbarali.maths@gmail.com, tahiridrees32@gmail.com }
}

\date{Submitted on January 2, 2018}

\begin{document}

\maketitle

\begin{abstract}

The general sum-connectivity index of a graph $G$ is defined as $\chi_{\alpha}(G)= \sum_{uv\in E(G)} (d_u + d_{v})^{\alpha}$ where $d_{u}$ is degree of the vertex $u\in V(G)$, $\alpha$ is a real number different from $0$ and $uv$ is the edge connecting the vertices $u,v$. In this note, the problem of characterizing the graphs having extremum $\chi_{\alpha}$ values from a certain collection of polyomino chain graphs is solved for $\alpha<0$. The obtained results together with already known results (concerning extremum values of polyomino chain graphs) give the complete solution of the aforementioned problem.


\end{abstract}
%
%
%
%
%
%
%
%
\section[Introduction]{Introduction}

All graphs considered in this note are simple, finite and connected. Those notations and terminologies from graph theory which are not defined here can be found in the books \cite{Harary69,Bondy08}.

The connectivity index (also known as Randi\'{c} index and branching index) is one of the most studied graph invariants, which was introduced in 1975 within the study of molecular branching \cite{r3}. The connectivity index for a graph $G$ is defined as
\[R(G)=\displaystyle\sum_{uv\in E(G)}(d_{u}d_{v})^{-\frac{1}{2}},\]
where $d_{u}$ represents the degree of the vertex $u\in V(G)$ and $uv$ is the edge connecting the vertices $u,v$ of $G$. Detail about the mathematical properties of this index can be found in the survey \cite{Li08}, recent papers \cite{Ali17R,Cui17,Li16,Mansour17,Das-17,Gutman-18,An2} and related references contained therein.

Several modified versions of the connectivity index were appeared in literature. One of such versions is the sum-connectivity index \cite{Zhou09}, which is defined as
\[\chi(G)=\sum_{uv\in E(G)}(d_{u}+d_{v})^{-\frac{1}{2}}.\]
Soon after the appearance of sum-connectivity index, its generalized version was proposed \cite{Zhou10}, whose definition is given as
\[\chi_{\alpha}(G)=\sum_{uv\in E(G)}(d_{u}+d_{v})^{\alpha},\]
where $\alpha$ is a non-zero real number. In this note, we are concerned with the general sum-connectivity index $\chi_{\alpha}$. Details about $\chi_{\alpha}$ can be found in the recent papers \cite{Zhu16,Cui-17,Tomescu16,Ali17,Ali-18,Wang17,Ramane17,Akhter17,Arshad17,Jamil17} and related references listed therein. We recall that $2\chi_{-1}(G)=H(G)$, where $H$ is the harmonic index \cite{Fajtlowicz87}, and $\chi_{1}$ coincides with the first Zagreb index \cite{Gutman72}, whose mathematical properties can be found in the recent surveys \cite{Borovicanin-17,Borovicanin-MCM19,Ali18} and related references cited therein.

A polyomino system is a connected geometric figure obtained by concatenating congruent squares side to side in a plane in such a way that the figure divides the plane into one infinite (external) region and a number of finite (internal) regions, and all internal regions must be congruent squares. For possible applications of polyomino systems, see, for example, \cite{Liu-Ou-13,Zhang-13,Gutman-99,Harary-97} and related references mentioned therein.
Two squares in a polyomino system are adjacent if they share a side. A polyomino chain is a polyomino system in which every square is adjacent to at most two other squares. Every polyomino chain can be represented by a graph known as polyomino chain graph. For the sake of simplicity, in the rest of this note, by the term \textit{polyomino chain} we always mean \textit{polyomino chain graph}.


The problem of characterizing graphs having extremum $\chi_{\al}$ values over the collection of certain polyomino chains, with fixed number of squares, was solved in \cite{z,AA2,y2} for $\al=1$. The results established in \cite{deng} give a solution of the aforementioned problem for $\al=-1$. An and Xiong \cite{An18} solved this problem for $\al>1$. While, the same problem was also addressed in \cite{Ali-2016} and its solution for the case $0<\al<1$ was reported there. The main purpose of the present note is to give the solution of the problem under consideration for all remaining values of $\al$, that is, for $\al<-1$ and $-1<\al<0$.

\section{Main Results}

Before proving the main results, we recall some definitions concerning polyomino chains. In a polyomino chain, a square adjacent with only one (respectively two) other square(s) is called terminal (respectively non-terminal) square. A kink is a non-terminal square having a vertex of degree 2. A polyomino chain without kinks is called \textit{linear chain}. A polyomino chain consisting of only kinks and terminal squares is known as zigzag chain. A segment is a maximal linear chain in a polyomino chain, including the kinks and/or terminal squares at its ends. The number of squares in a segment $S_r$ is called its length and is denoted by $l(S_r)$ (or simply by $l_r$). If a polyomino chain $B_{n}$ has segments $S_{1}, S_{2},...,S_{s}$ then the vector $(l_{1},l_{2},...,l_{s})$ is called length vector of $B_n$. A segment $S_r$ is said to be external (internal, respectively) segment if $S_r$ contains (does not contain, respectively) terminal square.
\begin{definition}\cite{z}
For $2\leq i\leq s-1$ and $1\leq j\leq s$,
$$\alpha_{i} =
\begin{cases}
1 & \text{if } l_i=2\\
0 & \text{if } l_i\geq3
\end{cases}$$
$$\beta_{j}=
\begin{cases}
1 & \text{if } l_j=2\\
0 & \text{if } l_j\geq3
\end{cases}$$
and $\alpha_{1}=\alpha_{s}=0.$
\end{definition}
Let $\Omega_{n}$ be the collection of all those polyomino chains, having $n$ squares, in which no internal segment of length 3 has edge connecting the vertices of degree 3.

\begin{thm}\label{t1}
\cite{Ali-2016} Let $B_{n}\in\Omega_{n}$ be a polyomino chain having $s$ segment(s) $S_{1}, S_{2},S_{3},...,S_{s}$ with the length vector $(l_{1},l_{2},...,l_{s})$. Then,
\begin{eqnarray*}
\chi_\alpha(B_{n})&=&3\cdot 6^\alpha n+(2\cdot5^\alpha-6^{\alpha+1}+4\cdot 7^\alpha)s
+(2\cdot 4^\alpha+2\cdot 5^\alpha+6^\alpha-4\cdot 7^\alpha)\\
&&+ \ (2\cdot6^\alpha-5^\alpha- 7^\alpha)[\beta_{1}+\beta_{s}]+ \ (5\cdot 6^\alpha-2\cdot 5^\alpha-4\cdot 7^\alpha+ 8^\alpha)\sum_{i=1}^{s}\alpha_{i}.
\end{eqnarray*}
\end{thm}

Let
$$f(\al)=2\cdot5^\alpha-6^{\alpha+1}+4\cdot 7^\alpha, \ g(\al)= 2\cdot6^\alpha-5^\alpha- 7^\alpha,$$
$$h(\al)=5\cdot 6^\alpha-2\cdot 5^\alpha-4\cdot 7^\alpha+ 8^\alpha.$$
Furthermore, let $\Psi_{\chi_\alpha}(S_{1})=f(\al)+g(\al)\beta_{1}, \ \Psi_{\chi_\alpha}(S_{s})=f(\al)+g(\al)\beta_{s}$ and for $s\geq3$, assume that $\Psi_{\chi_\alpha}(S_{i})=f(\al)+h(\al)\alpha_{i}$ where $2\leq i\leq s-1$. Then
\begin{equation}\label{Eq.11}
\Psi_{\chi_\alpha}(B_{n})=\sum_{i=1}^{s}\Psi_{\chi_\alpha}(S_{i})=f(\al)s+g(\al)(\beta_{1}+\beta_{s})
+h(\al)\sum_{i=1}^{s}\alpha_{i} \ .
\end{equation}
Hence, the formula given in Theorem \ref{t1} can be rewritten as
\begin{equation}\label{Eq.12}
\chi_\alpha(B_{n})=3\cdot 6^\alpha n +(2\cdot 4^\alpha+2\cdot 5^\alpha+6^\alpha-4\cdot 7^\alpha)+\Psi_{\chi_\alpha}(B_{n}).
\end{equation}
The next lemma is a direct consequence of the relation (\ref{Eq.12}).

\begin{lem}\label{L1}
\cite{Ali-2016}
For any polyomino chain $B_{n}$ having $n\geq3$ squares, $\chi_\alpha(B_{n})$ is maximum (respectively minimum) if and only if $\Psi_{\chi_\alpha}(B_{n})$ is maximum (respectively minimum).
\end{lem}

Lemma \ref{L1} will play a vital role in proving the main results of the present note.

\begin{lem}\label{t22}
\cite{Ali-2016}
Let $B_{n}\in\Omega_{n}$ be a polyomino with $n\geq3$ squares. If $f(\al)$, $f(\al)+2g(\al)$ and $f(\al)+2h(\al)$ are all negative, then
\[\chi_\alpha(Z_{n})\leq \chi_\alpha(B_{n})\leq \chi_\alpha(L_{n}).\]
Right (respectively left) equality holds if and only if $B_{n}\cong L_{n}$ (respectively $B_{n}\cong Z_{n}$).
\end{lem}

\begin{prop}\label{p-1}
Let $B_{n}\in\Omega_{n}$ be a polyomino chain having $n\geq3$ squares. Let $x_0\approx -3.09997$ be a root of the equation $f(\al)=0$. Then, for $x_0 < \alpha <0 $, it holds that
$$
\chi_\alpha(Z_{n}) \leq \chi_\alpha(B_{n})\leq \chi_\alpha(L_{n}),
$$
with right (respectively left) equality if and only if $B_{n}\cong L_{n}$ (respectively $B_{n}\cong Z_{n}$).
\end{prop}

\begin{proof}
It can be easily checked that $f(\al)$, $f(\al)+2g(\al)$ and $f(\al)+2h(\al)$ are negative for $x_0<\alpha<0$, and
hence, from Lemma \ref{t22}, the required result follows.
\end{proof}

\begin{prop}\label{p-2}
Let $B_{n}\in\Omega_{n}$ be a polyomino with $n\geq3$ squares. Let $x_0\approx -3.09997$ be a root of the equation $f(\al)=0$. Then, for $\alpha\le x_0$, the following inequality holds
\[\chi_\alpha(B_{n}) \ge \chi_\alpha(Z_{n}),\]
with equality if and only if $B_{n}\cong Z_{n}$.
\end{prop}

\begin{proof}
We note that $f(\al)$ is non-negative and both $g(\al)$, $h(\al)$ are negative for $\alpha\le x_0 \approx -3.09997$.
Suppose that the polyomino chain $B_{n}^{*}\in\Omega_{n}$ has the minimum $\Psi_{\chi_\alpha}$ value for $\alpha\le x_0$. Further suppose that $S_{1}, S_{2},...,S_{s}$ be the segments of $B_{n}^{*}$ with the length vector $(l_{1},l_{2},...,l_{s})$. It holds that
$$\Psi_{\chi_\alpha}(Z_{n}) = 2f(\al) + 2g(\al) + (n-3)(f(\al) + h(\al)) \le 2f(\al) + 2g(\al) < f(\al) = \Psi_{\chi_\alpha}(L_{n}),$$
which implies that $s\ge2$.

If at least one of external segments of $B_{n}^{*}$ has length greater than 2. Without loss of generality, assume that $l_{1}\geq3$. Then, there exist a polyomino chain $B_{n}^{(1)}\in\Omega_{n}$ having length vector $(\underbrace{2,2,\cdots,2}_{(l_1-1)-times},l_{2},...,l_{s})$ and
\[\Psi_{\chi_\alpha}(B_{n}^{(1)})-\Psi_{\chi_\alpha}(B_{n}^{*})= g(\al) + \left(l_1 -2\right)(f(\al) + h(\al))\le f(\al) + g(\al) + h(\al)<0,\]
for $\alpha\le x_0 \approx -3.09997$, which is a contradiction to the definition of $B_{n}^{*}$. Hence both external segments of $B_{n}^{*}$ must have length 2.

If some internal segment of $B_{n}^{*}$ has length greater than 2, say $l_{j}\geq3$ where $2\leq j\leq s-1$ and $s\geq3$. Then, there exists a polyomino chain $B_{n}^{(2)}\in\Omega_{n}$ having length vector $(l_{1},l_{2},...,l_{j-1},2,l_{j}-1,...,l_{s})$ and
\[\Psi_{\chi_\alpha}(B_{n}^{(2)})-\Psi_{\chi_\alpha}(B_{n}^{*})= f(\al) + (1+y)h(\al)<0, \ \ \ (\text{where $y=0$ or 1})\]
for $\alpha\le x_0 \approx -3.09997$, which is again a contradiction. Hence, every internal segment of $B_{n}^{*}$ has length 2.

\noindent
Therefore, $B_{n}^{*}\cong Z_{n}$ and from Lemma \ref{L1}, the desired result follows.
\end{proof}

\begin{prop}\label{p-3}
Let $B_{n}\in\Omega_{n}$ be a polyomino with $n\geq3$ squares. Let $\alpha\approx -3.09997$ be a root of the equation $f(\al)=0$. Then, the following inequality holds
\[\chi_\alpha(B_{n}) \le 3\cdot 6^\alpha n +(2\cdot 4^\alpha+2\cdot 5^\alpha+6^\alpha-4\cdot 7^\alpha),\]
with equality if and only if $B_{n}$ does not contain any segment of length 2.
\end{prop}

\begin{proof}
From Equation (\ref{Eq.11}), it follows that
\begin{equation*}
\Psi_{\chi_\alpha}(B_{n})=g(\al)(\beta_{1}+\beta_{s})+h(\al)\sum_{i=1}^{s}\alpha_{i} \le 0 .
\end{equation*}
Clearly, the equality $\Psi_{\chi_\alpha}(B_{n})=0$ holds if and only if $B_{n}$ does not contain any segment of length 2. Hence, by using Lemma \ref{L1}, we have the required result.
\end{proof}

Let $\mathcal{Z}_{n}^*$ be a subclass of $\Omega_{n}$ consisting of those polyomino chains which do not contain any segment of length equal to 2 or greater than 4, and contain at most one segment of length 4. Let $\mathcal{Z}_{n}$ be a subclass of $\Omega_{n}$ consisting of those polyomino chains in which every internal segment (if exists) has length 3 or 4, every external segment has length at most 4, at most one external segment has length 2, at most one segment has length 4 and if some internal segment has length 4 then both the external segments have length 3. Let $Z_{n}^\dag\in\Omega_{n}$ be the polyomino chain in which every internal segment (if exists) has length 3, every external segment has length at most 3 and at most one external segment has length 2.

\begin{prop}\label{p-4}
Let $B_{n}\in\Omega_{n}$ be a polyomino with $n\geq3$ squares. Let $x_0 \approx -3.09997$ and $x_1 \approx -5.46343$ be the roots of the equations $f(\al)=0$ and $f(\al)+g(\al)=0$, respectively. Then, for $x_1 < \alpha < x_0$, the following inequality holds
\begin{equation}\label{Eq-Z1}
\chi_\alpha(B_{n}) \le \chi_\alpha\left( Z_{n}^* \right),
\end{equation}
with equality if and only if $B_n \cong Z_{n}^*\in \mathcal{Z}_{n}^*$. Also, for $\alpha = x_1$, the following inequality holds
\begin{equation}\label{Eq-Z3}
\chi_\alpha(B_{n}) \le \chi_\alpha\left( Z_{n}^\maltese \right),
\end{equation}
with equality if and only if $B_n \cong Z_{n}^\maltese \in \mathcal{Z}_{n}$.  Furthermore, for $\alpha < x_1$, the following inequality holds
\begin{equation}\label{Eq-Z2}
\chi_\alpha(B_{n}) \le \chi_\alpha\left( Z_{n}^\dag \right),
\end{equation}
with equality if and only if $B_n \cong Z_{n}^\dag$.
\end{prop}

\begin{proof}
For $n=3$, the result is obvious. We assume that $n\ge4$. It can be easily checked that $f(\al)$ is positive and both $g(\al)$, $h(\al)$ are negative for $x_1 < \alpha < x_0$.
Suppose that for the polyomino chain $B_{n}^{*}\in\Omega_{n}$, $\Psi_{\chi_\alpha}(B_{n}^{*})$ is maximum for $\alpha < x_0$. Let $B_{n}^{*}$ has $s$ segments $S_{1}, S_{2},...,S_{s}$ with the length vector $(l_{1},l_{2},...,l_{s})$.

If $s\ge 3$ and at least one of internal segments of $B_{n}^{*}$ has length 2, say $l_{i}=2$ for $2\le i \le s-1$, then there exists a polyomino chain $B_{n}^{(1)}\in\Omega_{n}$ having length vector
\[
\begin{cases}
(l_1,l_{2},..., l_{s-1} + l_{s}-1)  & \text{ if } i=s-1, \\
(l_1,l_{2},..., l_{i-1}, l_i + l_{i+1}-1, l_{i+2},...,l_{s})  & \text{ otherwise},
\end{cases}
\]
and
\[\Psi_{\chi_\alpha}(B_{n}^{*})-\Psi_{\chi_\alpha}(B_{n}^{(1)})=
\begin{cases}
f(\al) + x\cdot g(\al) + h(\al)<0  & \text{ if } i=s-1, \\
f(\al) + (1+ y)h(\al)<0 & \text{ otherwise},
\end{cases}
\]
for $\alpha < x_0$, where $x,y\in \{0,1 \}$. This is a contradiction. Hence, every internal segment (if exists) of $B_{n}^{*}$ has length greater than 2.

If at least one of segments of $B_{n}^{*}$ has length greater than 4, say $l_{i}\ge5$ for $1\le i \le s$, then there exists a polyomino chain $B_{n}^{(2)}\in\Omega_{n}$ having length vector
\[
\begin{cases}
(3,l_1-2,l_{2},l_3,...,l_{s})  & \text{ if } i=1, \\
(l_1,l_{2},..., l_{i-1}, 3, l_i - 2, l_{i+1}, l_{i+2},...,l_{s})  & \text{ if } 2\le i \le s-1,\\
(l_1,l_{2},..., l_{s-1}, 3, l_s - 2)  & \text{ if } i = s,
\end{cases}
\]
and
\[\Psi_{\chi_\alpha}(B_{n}^{*})-\Psi_{\chi_\alpha}(B_{n}^{(2)})= -f(\al) <0,\]
a contradiction. Hence, every segment of $B_{n}^{*}$ has length less than than 5.

If at least two segments of $B_{n}^{*}$ have length 4, say $l_i=l_j=4$ for $1\le i,j \le s$, then there exists a polyomino chain $B_{n}^{(3)}\in\Omega_{n}$ having length vector
$(3,l_1,l_{2},..., l_{i-1}, l_{i}-1, l_{i+1},...,l_{j-1}, l_{j}-1, l_{j+1},...,l_{s})$ and
\[\Psi_{\chi_\alpha}(B_{n}^{*})-\Psi_{\chi_\alpha}(B_{n}^{(3)})= -f(\al) <0,\]
a contradiction. Hence, $B_{n}^{*}$ contains at most one segment of length 4.

If both the external segments of $B_{n}^{*}$ have length 2, then ($s\ge3$ because $n\ge4$) there exists a polyomino chain $B_{n}^{(4)}\in\Omega_{n}$ having length vector
$(l_1+1,l_{2},l_3,...,l_{s-1})$ and
\[\Psi_{\chi_\alpha}(B_{n}^{*})-\Psi_{\chi_\alpha}(B_{n}^{(4)})= f(\al) + 2g(\al)<0,\]
which is again a contradiction. Hence, at most one external segment has length 2. In what follows, without loss of generality, we assume that $l_s=2$ whenever some external segment has length 2.

If some external segment of $B_{n}^{*}$ has length greater 2, say $l_{1}=2$, then there exists a polyomino chain $B_{n}^{(5)}\in\Omega_{n}$ having length vector $(l_{2}+1, l_{3}, l_{4},...,l_{s})$ and
\[\Psi_{\chi_\alpha}(B_{n}^{*})-\Psi_{\chi_\alpha}(B_{n}^{(5)})= f(\al) + g(\al)<0, \ \text{ (because $l_2\ge3$)}\]
for $x_1 < \alpha < x_0$, which is again a contradiction. Hence, if $\Psi_{\chi_\alpha}(B_{n}^{*})$ is maximum for $x_1<\alpha < x_0$ then every external segment of $B_{n}^{*}$ has length greater than 2. Therefore, if $\Psi_{\chi_\alpha}(B_{n}^{*})$ is maximum for $x_1<\alpha < x_0$ then $B_{n}^{*}\cong Z_{n}^*$ and thence from Lemma \ref{L1}, inequality (\ref{Eq-Z1}) follows.

In the remaining proof, we assume $\alpha \le x_1$.

If $B_{n}^{*}$ contains a segment of length 4, say $l_{i}=4$ for $1\le i \le s$, then there exists a polyomino chain $B_{n}^{(6)}\in\Omega_{n}$ having length vector
\[
\begin{cases}
(2,l_1-1,l_{2},l_3,...,l_{s})  & \text{ if } i=1, \\
(l_1,l_{2},..., l_{i-1}, l_i -1 , l_{i+1}, l_{i+2},...,l_{s}+1)  & \text{ if } 2\le i \le s-1 \text{ and } l_s=2,\\
(l_1,l_{2},..., l_{i-1}, l_i -1 , l_{i+1}, l_{i+2},...,l_{s},2)  & \text{ if } 2\le i \le s-1 \text{ and } l_s=3,\\
(l_1,l_{2},..., l_{s-1}, l_s - 1,2)  & \text{ if } i = s,
\end{cases}
\]
and
\[\Psi_{\chi_\alpha}(B_{n}^{*})-\Psi_{\chi_\alpha}(B_{n}^{(6)})=
\begin{cases}
g(\al)  & \text{ if } 2\le i \le s-1 \text{ and } l_s=2,\\
-f(\al) - g(\al)  & \text{ otherwise}.
\end{cases}
\]
This last equation together with the fact that for $\alpha < x_1$, both $g(\al)$ and $-f(\al) - g(\al)$ are negative, gives a contradiction. The same equation together with the fact that for $\alpha = x_1$, only $g(\al)$ is negative, arises also a contradiction if $2\le i \le s-1$ and  $l_s=2$. Therefore, if $\Psi_{\chi_\alpha}(B_{n}^{*})$ is maximum for $\alpha < x_1$ then $B_{n}^{*}\cong Z_{n}^\dag$ and if $\Psi_{\chi_\alpha}(B_{n}^{*})$ is maximum for $\alpha = x_1$ then $B_{n}^{*} \in \mathcal{Z}_{n}$, and thence from Lemma \ref{L1}, inequalities (\ref{Eq-Z3}) and (\ref{Eq-Z2}) follow.
\end{proof}

Propositions \ref{p-1}, \ref{p-2}, \ref{p-3} and \ref{p-4}, together with the already reported results in \cite{z,AA2,y2,deng, An18,Ali-2016}, yield Table \ref{table-1} which gives information about the polyomino chains having extremum $\chi_\alpha$ values in the collection $\Omega_{n}$ for $n\ge3$.


\begin{table}[ht]
\begin{center}
\footnotesize
  \begin{tabular}{ | c | c | c | }\hline

                  & Polyomino Chain(s) with Maximal $\chi_\alpha$ Value &  Polyomino Chain(s) with Minimal $\chi_\alpha$ Value \\ \hline
$\alpha >0 $      & $Z_n$                                               &   $L_n$                                              \\ \hline
$x_0 <\alpha <0$  & $L_n$                                               &   $Z_n$                                               \\ \hline
$\alpha = x_0$    & chains having no segment of length 2                &   $Z_n$                                               \\ \hline
$x_1 <\alpha<x_0$ & members of $\mathcal{Z}_{n}^*$                      &   $Z_n$                                               \\ \hline
$\alpha= x_1$     & members of $\mathcal{Z}_{n}$                        &   $Z_n$                                                \\ \hline
$\alpha < x_1$    & $Z_{n}^\dag$                                        &   $Z_n$                                                \\ \hline

\end{tabular} \vspace{-6mm}
\end{center}
\caption{Polyomino chains having extremum $\chi_\alpha$ values in the collection $\Omega_{n}$ for $n\ge3$.}\label{table-1}
\end{table}

\end{document}